\documentclass[fleqn,10pt]{SelfArx} 

\usepackage[english]{babel} 

\usepackage{lipsum} 

\definecolor{color1}{RGB}{0,0,90} 
\definecolor{color2}{RGB}{0,20,20} 

\usepackage{hyperref} 

\hypersetup{
	hidelinks,
	colorlinks,
	breaklinks=true,
	urlcolor=color2,
	citecolor=color1,
	linkcolor=color1,
	bookmarksopen=false,
	pdftitle={Title},
	pdfauthor={Author},
}

\usepackage{amssymb,amsthm,amsmath,}
\newtheorem{thm}{Theorem}[section]

\newtheorem{lem}[thm]{Lemma}

\newtheorem{prop}[thm]{Proposition}

\JournalInfo{Journal, Vol. XXI, No. 1, 1-5, 2013} 
\Archive{Additional note} 

\PaperTitle{Weighted Composition Operator on Fock Space} 

\Authors{Rui Hu\textsuperscript{1}*} 
\affiliation{\textsuperscript{1}\textit{Department of Mathematics and Statistics, University of Chongqing, Chongqing, People's Republic of China}} 
\affiliation{*\textbf{Corresponding author}: 202106021013@stu.cqu.edu.cn} 

\newcommand{\keywordname}{Keywords} 

\Abstract{ In this thesis, we establish a necessary and sufficient condition for a weighted composition operator to commute  with a self-adjoint weighted composition operator on the Fock space, then obtain a sufficient condition for these commuting weighted composition operators to be  normal.  }

\begin{document}

\maketitle 

\tableofcontents 

\thispagestyle{empty} 


\section{Introduction} 


The operator theory in function space is closely related to many fields in mathematics, and the study of it is helpful to promote the development of other disciplines. The weighted composition operator is a combination of composition operator and multiplication operator. The relationship between the properties of operators and the properties of their defined functions is reflected. On the one hand, the conclusion of holomorphic function theory can be applied to the properties of linear operators. On the other hand, operator theory can also be used as a tool to study holomorphic function space, so the relationship between operator theory and function theory can be established, which is one of the important contents of operator theory.

Let $\mathbb D$ denote the open unit disk in the complex plane $\mathbb C$. For any positive parameter $\alpha$, we consider the Ganssian measure
$$d\lambda_{\alpha}(z)=\frac{\alpha}{\pi}e^{-\alpha|z|^2}dA(z),$$
where $dA$ is the Euclidean area measure on the complex plane and $d\lambda_{\alpha}$ is a probability measure. The Fock space $F_{\alpha}^2$ consists of all entire functions $f$ in $L^2(\mathbb C,d\lambda_{\alpha})$ and $F_{\alpha}^2$ is a Hilbert space with the following inner product inherited from $L^2(\mathbb C,d\lambda_{\alpha})$:
$$ \langle  f,g \rangle =\int_{\mathbb C}f(z)\overline{g(z)}d\lambda_{\alpha}(z).$$
More on the Fock space can be found at \cite{zhuf}.

In addition, for example, Zhao Liankuo studied unitary weighted composition operators on $\mathbb C^n $. We characterize the normal weighted composition operators on the Fock space $\mathbb C^N $, the self-adjoint weighted composition operators on the Fock space $\mathbb C^N $ are discussed. We also study the isometrically weighted composition operators on the Fock space $\mathbb C^N $. It is proved that the isometrically weighted composition operator on the Fock space $\mathbb C^N $ is unitary operator, see \cite{zf,zf2,zf1}

In addition to the classical analytic function space, in recent years, some scholars have begun to pay attention to the weighted composition operator on the function space with the tree as the domain. And try to generalize the related conclusion on the traditional function space to this kind of new function space, see \cite{ap}.

It is worth noting that Ko studied weighted composition operators on Hardy spaces $H^2 $ in \cite{koe} in 2021, and the necessary and sufficient conditions for a weighted composition operator to be interchangeable with another self-adjoint weighted composition operator are obtained. Furthermore, it is proved that all the weighted composition operators which are commutative with a class of self-adjoint weighted composition operators are normal.
Inspired by these works, we extend some of the results in \cite{koe} to Fock spaces. We establish a necessary and sufficient condition that a weighted composition operator and another self-adjoint weighted composition operator on~Bergman~space are commutative. The sufficient conditions for the normality of this class of weighted composition operators on Fock spaces are given.


\section{Preliminaries}

In this section, we recall some definitions needed for our program. In this paper, we consider the Fock space $F^2$.

Let $\mathcal{B}(\mathcal{H})$ be the algebra of all bounded linear operators on a separable complex Hilbert
space $\mathcal{H}$. Given a fixed operator $T \in \mathcal{B}(\mathcal{H})$, we say an operator $S$ commutes with $T$ if $TS=ST$. The set of all operators which commute with $T$, denoted $\{T\}'$.

If $\varphi$ is an analytic mapping from $\mathbb C$ into itself, for all $h \in F^2$, the composition operator $C_{\varphi}$ on $F^2$ is defined by 
$$C_{\varphi}h = h\cdot\varphi,$$
The composition operator $C_{\varphi}$ is bounded on $F^2$ if and only if 
$$\varphi(z)=az+b,\ \ z\in \mathbb C,$$
where $a,b\in \mathbb C$ and $|a|\leq1$ (see \cite{car}).
For a non-constant valued function $\psi $ in $\mathbb C $, the operator $W_{\psi, \varphi}$ given by
$$W_{\psi, \varphi}f(z)=\psi(z)f(\varphi(z)), \ \ \  f\in F_{\alpha}^2, \ z\in \mathbb C$$
is called a weighted composition operator.

For each $\beta\in\mathbb C$, the function $K_z(w)=e^{\alpha \overline{z}w},w\in \mathbb C$ called the reproducing kernel for $F^2$ at $\beta$, has the property that
$$f(z)=\langle f, K_z\rangle, \ \ \ \  f\in F_{\alpha}^2.$$

For the product of two weighted composition operators on Fock space, we have the following simple conclusion.
\begin{lem}\label{wf}
If $h\in F^2$, then for weighted composition operator $W_{f,\varphi}$ and $W_{g, \psi}$,we have
$$W_{f, \varphi} W_{g, \psi}=M_{f \cdot(g \circ \varphi)} C_{\psi\circ\varphi}.$$
Here and hereafter, we use $M_h $ to denote the multiplication operator with $h$ as the sign.
\end{lem}
\begin{proof}
According to the definition of weighted composition operator, we have
\begin{align*}
W_{f, \varphi} W_{g, \psi}(h) & =M_{f} C_{\varphi} M_{g} C_{\psi}(h)\\
&=M_{f} C_{\varphi}(g \cdot(h \circ \psi)) \\
& =f \cdot(g \circ \varphi) \cdot(h \circ(\psi \circ \varphi)) \\
& =\left(M_{f \cdot(g \circ \varphi)} C_{\psi\circ\varphi}\right)(h),
\end{align*}
\end{proof}

The following describes the results of the adjoint operator of the weighted composition operator on the Fock space acting on the reproducing kernel, which will be used several times in the following.
\begin{lem}\label{W*}
Let $W_{f,\varphi}$ is a weighted composition operator on $F^2$, then $W_{f,\varphi}^*K_{z}=\overline{f(z)}K_{\varphi(z)}$.
\end{lem}
\begin{proof}
For $h\in F^2$,
$$\langle h,W_{f,\varphi}^*K_{z}\rangle =\langle f\cdot (h\circ \varphi),K_{z}\rangle=f(z)h(\varphi(z))=\langle h,\overline{f(z)}K_{\varphi(z)}\rangle.$$
Thus $W_{f,\varphi}^*K_{z}=\overline{f(z)}K_{\varphi(z)}$.
\end{proof}

\section{Main Results}

In this section, we investigate weighted composition operators commuting with self-adjoint weighted composition operators on $F^2$ . In particular, we show that weighted composition operators commuting with some self-adjoint weighted composition operators on $F^2$  are normal. We assume that $\varphi$ is not an identity map.

\begin{thm}\label{thm5.1}
   Let $f\in H^\infty$ and let $\varphi$ be an analytic map of the unit disk into itself.If $W_{f,\varphi}$ is self-adjoint on $F^2$,then $f(0)$ and $\varphi'(0)$are real,and
$$f(z)=ce^{\overline{a}_0 z}, \ \varphi(z)=a_0+a_{1}z.$$
where $a_0=\varphi(0)$, $a_1=\varphi'(0)$, $c=f(0)$;
Conversely,let $a_0\in\mathbb{D}$,and let $c$ and $a_1$ be real numbers.
If
$$f(z)=ce^{\overline{a}_0z},\varphi(z)=a_0+a_{1}z,$$
then $W_{f,\varphi}$ is self-adjoint.
\end{thm}
\begin{proof}
The reproducing kernel function on Fock space $F^2$ is $K_{\beta}z=e^{z\overline{\beta}}$, then
$$W_{f,\varphi}K_\beta(z)=T_fC_{\varphi}e^{\overline{\beta}z}=f(z)e^{\varphi(z)\overline{\beta}},$$
and
$$W_{f,\varphi}^{*}K_\beta(z)=\overline{f(\beta)}e^{\overline{\varphi(\beta)}z}.$$
Thus $W_{f,\varphi}$ is self-adjoint if and only if
\begin{equation}\label{5-1}
f(z)e^{\varphi(z)\overline{\beta}}=\overline{f(\beta)}e^{\overline{\varphi(\beta)}z}.
\end{equation}

In particular, letting $\beta=0$ in (\ref{5-1}),we get
$$f(z)=\overline{f(0)} e^{\overline{\varphi(0)}z}.$$
for all $z$ in disk. Setting $z=0$,we get $f(0)=\overline{f(0)}$,
so that $f(0)\in\mathbb{R}$.Defining $c=f(0)$ and $a_0=\varphi(0)$,we can write $f$ as
\begin{equation}
\label{5-2}
f(z)=ce^{\overline{a}_0z}.
\end{equation}
Combining (\ref{5-1}) and (\ref{5-2}) we get
$$f(z)e^{\varphi(z)\overline{\beta}}=ce^{\overline{a}_0z}e^{\varphi(z)\overline{\beta}}$$
and
$$\overline{f(\beta)}e^{\overline{\varphi(\beta)}z}=ce^{a_0\overline{\beta}}e^{\overline{\varphi(\beta)}z},$$
which means $W_{f,\varphi}$ is self-adjoint if and only if
$$e^{\overline{a}_0z}e^{\varphi(z)\overline{\beta}}=e^{a_0\overline{\beta}}e^{\overline{\varphi(\beta)}z}.$$
Thus
\begin{equation}\label{5-3-1}
(\varphi(z)-a_0)\overline{\beta}=(\overline{\varphi(\beta)-a_0})z.
\end{equation}
Notice that the expression on the right is a polynomial of degree one in the variable $z$. This means that the expression on the left
$$(\varphi(z)-a_0)\overline{\beta}$$
must also be a polynomial of degree one in $z$.
Then take $\varphi(z)-a_0=a_1z$, we get
$$\varphi(z)=a_0+a_1z.$$
Since $\varphi$ is analytic, we have $a_1=\varphi'(0)$.In particular, combining the expression of $\varphi$ and (\ref{5-3-1}), we find $a_1=\overline{a}_1$, that is $a_1\in\mathbb R$.

Conversely, if $a_1\in\mathbb{R}$ and $a_0\in\mathbb{D}$ are such that $\varphi(z)=a_0+a_{1}z$ maps the disk into itself, and $f(z)=ce^{\overline{a}_0z}$ for $c\in\mathbb{R}$, then we have
\begin{align*}
    W_{f,\varphi}K_\beta(z)&=f(z)e^{\varphi(z)\overline{\beta}}\\
&=ce^{\overline{a}_0z}e^{(a_0+a_1z)\overline{\beta}}\\
&=ce^{(\overline{a}_0+a_1\overline{\beta})z+a_0\overline{\beta}}
\end{align*}
and
\begin{align*}
W_{f,\varphi}^*K_\beta(z)&=\overline{f(\beta)}e^{\overline{\varphi(\beta)}z}\\
&=ce^{a_0\overline{\beta}}e^{\overline{(a_0+a_1\beta})z}\\
&=ce^{(\overline{a}_0+a_1\overline{\beta})z+a_0\overline{\beta}}.
\end{align*}
So $W_{f,\varphi}K_\alpha=W_{f,\varphi}^*K_\alpha$ shows (\ref{5-1}) holds,which means $W_{f,\varphi}$ is self-adjoint.
\end{proof}

\begin{lem}\label{lm5.1}
    For $a_1\in\mathbb{R}$ and $a_0$ a complex number, let $\varphi(z)=a_0+a_{1}z$. Let $b$ be a fixed point in $\mathbb D$ of $\varphi$, and let $h(z)=(z-b)/(\bar{b}z-1)$,Then:\\
(1)$b=0$ when $a_0=0$ and 
$b=\frac{a_0}{1-a_1}$ for $a_0\neq0$~, ;\\
(2)~$h(\varphi(z))=\alpha h(z)$,where
$$\alpha=\frac{a_1(\bar{b}z-1)}{\bar{b}a_1z+\bar{b}a_0-1}.$$
\end{lem}
\begin{proof}
    (1) If $a_0=0$, then $a_1b=b$, since $a_1\neq1$,thus $b=0$.
    
If $a_0\neq0$,we have $a_0+a_1b=b$,
By a simple calculation, we can get $b=\frac{a_0}{1-a_1}$.\\
(2)According to $a_0+a_1b=b$,we have $a_1(z-b)=a_1z-a_1b=a_1z+a_0-b$,then:
\begin{align*}
h(\varphi(z))&=\frac{\varphi(z)-b}{\overline{b}\varphi(z)-1}\\
&=\frac{a_0+a_1z-b}{\overline{b}(a_0+a_1z)-1}\\
&=\frac{a_1z+a_0-b}{\overline{b}a_1z+\overline{b}a_0-1}\\
&=\frac{a_1(z-b)}{\overline{b}a_1z+\overline{b}a_0-1}\\
&=a_1\cdot \frac{z-b}{\overline{b}z-1} \cdot \frac{\overline{b}z-1}{\overline{b}a_1z+\overline{b}a_0-1}\\
&=\alpha\frac{z-b}{\overline{b}z-1}=\alpha h(z).
\end{align*}
where
$$\alpha=\frac{a_1(\overline{b}z-1)}{\bar{b}a_1z+\overline{b}a_0-1}.$$
Therefore $h(\varphi(z))=\alpha h(z)$.
\end{proof}

\begin{lem}\label{lm5.2}
   Let $a_1\in\mathbb{R}$ be real, then $\varphi(z)=a_0+a_1z$ maps the open unit disk into itself if and only if
$$
\begin{cases}
\left |a_0  \right | < 1;\\
-1+\left |a_0  \right |\le a_1\le 1-\left |a_0  \right |.
\end{cases}
$$
\end{lem}
\begin{proof}
    (1)Suppose $a_0$ be real ans non-negative, to show that
$$0\le a_0<1,-1+a_0 \le a_1\le 1-a_0.$$
Since $\varphi$ is a linear fractional map, Theorem 10 of \cite{cowl} says that $\varphi$ maps the open unit disk into itself if and only if $\varphi$ maps the open interval (-1,1) into itself. In particular, if $\varphi$ maps the open disk into itself, then $\varphi(0)=a_0$ must
lie in the open unit disk; that is $0\le a_0<1$.
Moreover, ${\varphi}'(z)=a_1$, so $\varphi$ is either increasing or decreasing on (-1,1), depending on the sign of $a_1$.
Thus $0\le a_0<1$, $\varphi$ maps the open unit disk into itself if and only if $\varphi(-1)$ and $\varphi(1)$ are in the $[-1,1]$.

Since
$$\varphi(-1)=a_0-a_1,\varphi(1)=a_0+a_1,$$
we see 
$$-1\le a_0-a_1\le 1,-1\le a_0+a_1\le 1,$$
thus $-1+a_0 \le a_1\le 1-a_0.$

(2)Choosing $\theta$ so that $\widetilde{a_0}(z)=a_0e^{i\theta}=|a_0|$, and let $\widetilde{\varphi}(z)=e^{i\theta}\varphi(e^{-i\theta}z)$, then we have $\widetilde{\varphi}(z)=\widetilde{a_0}+a_1z$ and $\widetilde{a_0}=a_0e^{i\theta}$,
then $\widetilde{\varphi}(z)$ satisfies the hypotheses of (1),
according to the result of (1), we get $\varphi(z)=a_0+a_1z$ maps the open unit disk into itself if and only if
$$\left |a_0  \right | < 1,-1+\left |a_0  \right |\le a_1\le 1-\left |a_0  \right |.$$
\end{proof}

\begin{lem}\label{lm5.3}
Let $\varphi(z)=a_0+a_1z$, where $\left | a_0 \right |<1$ and $-1+|a_0|\leq a_1<(1-|a_0|)$.
If $W_{f,\varphi}$ is self-adjoint, then $W_{f,\varphi}e_0=\overline{f(b)} e_0$, $W_{f,\varphi}e_j=\overline{f(b)}\alpha^{j}e_j$,
where
$$\alpha=\frac{a_1(\bar{b}z-1)}{\bar{b}a_1z+\bar{b}a_0-1}, \ \ \ \ e_j(z)=e^{\bar{b}z-\frac{1}{2}|b|^2}\left ( \frac{z-b}{\bar{b}z-1}  \right )^j,\ \ \ \  j=0, 1, \cdots,$$
$b\in \mathbb D$ is a fixed point of $\varphi$.
\end{lem}
\begin{proof}
If $j=0$, by lemma \ref{W*}
\begin{align*}
W_{f,\varphi}e_0&=W_{f,\varphi}^*e^{-\frac{1}{2}|b|^2}K_b=\overline{f(b)}e^{-\frac{1}{2}|b|^2}K_{\varphi(b)}\\
&=\overline{f(b)}e^{-\frac{1}{2}|b|^2}K_b=\overline{f(b)}e_0.
\end{align*}
If $j\geq1$, then $e_j=e_0h^j$,where $h$ defined by lemma \ref{lm5.1}.Then
\begin{align*}
W_{f,\varphi}e_j&=W_{f,\varphi}e_0h^j=W_{f,\varphi}(e_0)C_{\varphi}h^j\\
&=\big(\overline{f(b)}e_0 \big)\big ( h\circ \varphi \big )^j =\overline{f(b)}\alpha^j(e_0h^j)\\
&=\overline{f(b)}\alpha^je_j.
\end{align*}
Thus, each $e_j $ is an eigenvector of $W _ {f,\ varphi} $.
\end{proof}

Next we will characterize the total weighted composition operators that are commutative with $W_{f,\varphi}$ on the Fock space, 
where $\varphi$ is analytic self-mapping of nonelliptic automorphisms on $\mathbb{D}$.
\begin{lem}\label{lm5.4}
    Let $g\in H^\infty$ which has no zero on $\mathbb D$,and let ~$\psi$ be an analytic map of $\mathbb{D}$ into itself.Assume that $W_{g,\psi}\in \left\{W_{f,\varphi}\right\}'$ where $W_ {f,\varphi}$ is self-adjoint.If $\varphi(b)=b$, then $\psi(b)=b$.
\end{lem}
\begin{proof}
    By Theorem \ref{thm5.1}, $f(z)=ce^{\overline{a}_0z}$,~$\varphi(z)=a_0+a_1z$,
$c=f(0), \ a_0=\varphi(0), \ a_1=\varphi'(0)$ where $a_0\in\mathbb{D}$, $c\in \mathbb R, \ a_1\in\mathbb{R}$.

Below we divide $b$ is zero or not two cases to discuss.
  (1) When $b=0$, $f(z)=c=f(0)$, $\varphi(z)=a_1z$ and $a_1\neq 1$~(Otherwise $\varphi$ is identity).
  By Lemma \ref{wf}, we have
  $$W_{f, \varphi} W_{g, \psi}=M_{f \cdot(g \circ \varphi)} C_{\psi\circ\varphi},\ \ \ \
 W_{g, \psi}W_{f, \varphi}=M_{g \cdot(f \circ \psi)} C_{\varphi\circ\psi}.$$
 Since $W_{g,\psi}\in \left\{W_{f,\varphi}\right\}'$, $\psi(a_1z)=a_1\psi(z)$.
 Set $\psi(z)=\sum\limits_{k=0}^{\infty}\gamma_kz^k$, then we get $\gamma_0=a_1\gamma_0$, since$a_1\neq 1$,
 hence $\gamma_0=0$, thus $\psi(0)=0$.

(2) Assume $b\neq0$, by Lemma \ref{W*}
\begin{align*}
W_{f,\varphi}K_b= W_{f,\varphi}^*K_b=\overline{f(b)}K_b.
\end{align*}
Since $W_{g,\psi}\in \left\{W_{f,\varphi}\right\}'$,
\begin{align*}
W_{f,\varphi}^*W_{g,\psi}^*K_b
=W_{g,\psi}^*W_{f,\varphi}^*K_b
=\overline{f(b)}W_{g,\psi}^*K_b.
\end{align*}
On the other hand, $g$ has no zero, so
\begin{align*}
W_{g,\psi}^*K_b=\overline{g(b)}K_{\psi(b)}(z)=\overline{g(b)}e^{\overline{\psi(b)}z} \neq0.
\end{align*}
thus, $W_{g,\psi}^*K_b$ is an eigenvector of $W_{f,\varphi}^*=W_{f,\varphi}$ corresponding to an eigenvalue $\overline{f(b)}$. 
If $a_1=1-|a_0|$, then $|b| =1$,it is clear that $b$ is not in $\mathbb{D}$. 
If $a_1=-1+|a_0|$, then $\varphi(z)=a_0+(-1+|a_0|)z$.
By $\varphi(\psi(z))=\psi(\varphi(z))$
\begin{align*}
    \psi(\varphi(b))=\psi(b)=\varphi(\psi(b)),
\end{align*}
then we get $\psi(b)=\frac{a_0}{2-|a_0|}$, so $b=\frac{a_0}{1-a_1}=\frac{a_0}{2-|a_0|}$, thus $\psi(b)=b$.

Finally if $-1+{|a_0|}^2<a_1<{(1-|a_0|)}^2$,
we know from Lemma \ref{lm5.3}and $W_{f,\varphi}$ is self-adjoint,
that $\mathrm{ker}~[W_{f,\varphi}-f(b)]^*=\mathrm{span}~\{e_0\}$, where $e_0=e^{-\frac{1}{2}|b|^2}K_b$.
Thus, exist $\nu\in\mathbb{C}$ such that $W_{g,\psi}^*K_b=\nu K_b$.
then
$$\nu K_b=W_{g,\psi}^*K_b=\overline{g(b)}K_{\psi(b)},$$
since
\begin{align*}
\nu e^{\overline{b}z}=\overline{g(b)}e^{\overline{\psi(b)}z},
\end{align*}
Therefore $\psi(b)=b$.
\end{proof}

Next, we give a sufficient condition that $W_{f,\varphi}$ and $W_{g,\psi}$ can be commuted.
\begin{thm}\label{thm5.2}
 Let $g\in H^\infty$ which has no zero on $\mathbb D$, and let ~$\psi$ be an analytic map of $\mathbb{D}$ into itself.Assume that $W_ {f,\varphi}$ is self-adjoint and $\varphi$, not an elliptic automorphism, has a fixed point $b$ $\mathbb{D}$. Then $W_{g,\psi}\in \{W_{f,\varphi}\}^{'}$ if $W_{g,\psi}$ has the following symbol functions; for $b\neq0$
$$ \psi(z)=d_0+\frac{d_2z}{1-d_1z}\  and \ g(z)=g(b)\exp\left \{ \overline{b}\left (  z+\frac{-d_0+d_3z}{1-d_0z}\right )   \right \}$$
where $d_0=\frac{(\eta-1)b}{{|b|}^{2}\eta-1}$,
$d_1=\frac{(\eta-1)\overline{b}}{{|b|}^{2}\eta-1}$,
$d_2=\eta\frac{({|b|}^{2}-1)^2}{({|b|}^{2}\eta-1)^2}$, and $d_3=\frac{|b|^{2}-\eta}{{|b|}^{2}\eta-1}$ for some $\eta\in\mathbb{C}$ and 
for $b=0$
$$\psi(z)= z,\ \ g(z)=g(0).$$
The converse is not true.
\end{thm}
\begin{proof}
If $W_{f,\varphi}$ is self-adjoint on $F^2$, by theorem \ref{thm5.1}, $f(z)=ce^{\overline{a_0}z}$, and $\varphi(z)=a_0+a_{1}z$ where $c=f(0)$, $a_0=\varphi(0)$, $a_1=\varphi'(0)$; for $a_0\in\mathbb D$, $c\in \mathbb R$, $a_1\in \mathbb R$.

(1)Assume $b\neq0$.By lemma \ref{lm5.1}, $a_0\neq 0$ and 
$$W_{f,\varphi}W_{g,\psi}K_b=W_{g,\psi}W_{f,\varphi}K_b=\overline{f(b)}W_{g,\psi}K_b.$$
Since $g$ has no zero, we have $W_{g,\psi}K_b\neq 0$, so$W_{g,\psi}K_b$ is an eigenvector for $W_{f,\varphi}^*=W_{f,\varphi}$with an eigenvalue $\overline{f(b)}$.
Since $\varphi$, not an identity map or an elliptic automorphim, has a fixed
point b in D, it suffices to consider only the case when $-1+ |a_0 |\leq a_1< 1-|a_0 |$.

Since $\mathrm{ker}~(W_{f,\varphi}-f(b))^*=\mathrm{span}\{e_0\}$, where $e_0=e^{-\frac{1}{2}|b|^2}K_b$,
$$W_{g,\psi}K_b=\delta' e_0=\delta'e^{-\frac{1}{2}|b|^2}K_b:=\delta K_b.$$
So
\begin{align*}
W_{g,\psi}e_1&=W_{g,\psi}\left(e^{-\frac{1}{2}|b|^2}K_b({\frac{z-b}{\overline{b}z-1}})\right)\\
&=e^{-\frac{1}{2}|b|^2}M_g(K_b\circ\psi)\left({\frac{z-b}{\overline{b}z-1}}\circ\psi\right)\\
&=e^{-\frac{1}{2}|b|^2}(M_gC_{\psi}K_b)\left({\frac{z-b}{\overline{b}z-1}}\circ\psi\right)\\
&=\delta K_be^{-\frac{1}{2}|b|^2}\left({\frac{z-b}{\overline{b}z-1}}\circ\psi\right).
\end{align*}
By \ref{lm5.1},
$$W_{f,\varphi}W_{g,\psi}e_1=W_{g,\psi}W_{f,\varphi}e_1=\overline{f(b)}\alpha W_{g,\psi}e_1,$$
that $W_{g,\psi}e_1$ is an eigenvector of $W_{f,\varphi}$ corresponding to eigenvalue $\overline{f(b)}\alpha$. Noticing that
$$\mathrm{ker}(W_{f,\varphi}-\overline{f(b)}\alpha)=\mathrm{span}\{e_1\},$$
so exist $c_1\in \mathbb C$ such that
\begin{align*}
  \delta K_be^{-\frac{1}{2}|b|^2}\left ( {\frac{z-b}{\overline{b}z-1}}\circ\psi \right ) =c_1e_1=c_1K_be^{-\frac{1}{2}|b|^2}\frac{z-b}{\overline{b}z-1},
\end{align*}
denoting $\eta=\frac{c_1}{\delta}$, we have
\begin{align*}
  \frac{\psi(z)-b}{\overline{b}\psi(z)-1}=\eta\frac{z-b}{\overline{b}z-1}.
\end{align*}
Thus
\begin{align}\label{psiz}
  \psi(z)=\frac{(|b|^2-\eta)z+(\eta-1)b}{\overline{b}(1-\eta)z+(|b|^2\eta-1)}.
\end{align}
Setting $d_0=\psi(0)=\frac{(\eta-1)b}{{|b|}^{2}\eta-1}$, $d_1=\frac{(\eta-1)\overline{b}}{{|b|}^{2}\eta-1}$,
Now we want to find $d_2$, so that $\psi$ can be represented as 
$$\psi(z)=d_0+\frac{d_2z}{1-d_1z}.$$
After calculating we get $d_2=\eta\frac{({|b|}^{2}-1)^2}{({|b|}^{2}\eta-1)^2}$.
Finally, we calculate the symbol $g$.
Since $W_{g,\psi}K_b=\delta K_b$, we get 
$$g(z)e^{\overline{b}\psi(z)}=\delta e^{\overline{b}z}.$$
Taking $z=b$, then $\delta=g(b)$, so
\begin{align*}
g(z)=&g(b)~e^{\overline{b}\left ( z-\psi(z) \right ) }\\
=&g(b)~e^{\overline{b}z}~e^{\overline{b} \frac{d_0-d_0^2z+d_1z}{1-d_1z}} \\
=&g(b)~e^{\overline{b}z}~e^{\overline{b} \frac{-d_0+d_3z}{1-d_1z}},
\end{align*}
where $d_3=\frac{|b|^{2}-\eta}{{|b|}^{2}\eta-1}$.

(2)Assume $b=0$. Since $W_{g,\psi}K_0=g(0)K_0$,
$$W_{g,\psi}K_0(z)\cdot z=M_g(K_0\circ\psi)(z\circ\psi)=W_{g,\psi}K_0(z)\cdot\psi(z)=g(0)K_0(z)\psi(z).$$
Thus
$$g(z)\cdot z=g(0)\psi(z).$$
According to
$$g(0)K_0(z)=W_{g,\psi}K_0(z)=g(z)K_0(z),$$
we have $g(z)=g(0)$, therefore$\psi(z)=z$.

Conversely, assume that $W_{g,\psi}$ has the form given by theorem, taking $\varphi(z)=\frac{1}{2}+\frac{1}{4}z,\ f(z)=ce^{\frac{1}{2}z}$, let $\varphi(b)=b$, then $b=\frac{2}{3}$,
then
\begin{align*}
\psi(z)=\frac{(\frac{4}{9}-\eta)z+(\eta-1)\frac{2}{3}}{\frac{2}{3}(1-\eta)z+\frac{4}{9}\eta-1},
\end{align*}
So we have
\begin{align*}
  \varphi(\psi(z))&=\frac{1}{2}+\frac{1}{4}\psi(z)\\
&=\frac{1}{2}+\frac{(\frac{1}{9}-\frac{1}{4}\eta)z+\frac{1}{6}(\eta-1)}{\frac{1}{6}(1-\eta)z+\frac{1}{9}\eta-\frac{1}{4}}\\
&=\frac{(\frac{7}{36}-\frac{1}{3}\eta)z+\frac{2}{9}\eta-\frac{7}{24}}{\frac{1}{6}(1-\eta)z+\frac{1}{9}\eta-\frac{1}{4}}
\end{align*}
and
\begin{align*}
  \psi(\varphi(z))&=\frac{(\frac{4}{9}-\eta)\varphi(z)+(\eta-1)\frac{2}{3}}{\frac{2}{3}(1-\eta)\varphi(z)+\frac{4}{9}\eta-1}\\
&=\frac{(\frac{4}{9}-\eta)(\frac{1}{2}+\frac{1}{4}z)+(\eta-1)\frac{2}{3}}{\frac{2}{3}(1-\eta)(\frac{1}{2}+\frac{1}{4}z)+\frac{4}{9}\eta-1}\\
&=\frac{(\frac{1}{9}-\frac{1}{4}\eta)z+\frac{1}{6}\eta-\frac{4}{9}}{\frac{1}{6}(1-\eta)z+\frac{1}{9}\eta-\frac{2}{3}},
\end{align*}
Thus $\varphi(\psi(z))\neq\psi(\varphi(z))$,
\end{proof}

If we assume that $C_{\psi}$ is bounded on Fock space $F^2$ in theorem \ref{thm5.2}, then letting $\overline{b}(1-\eta)=0$ in (\ref{psiz}),since $b\neq 0$, $\eta=1$, we get $d_0=d_1=0$, $d_2=1$, thus $\psi(z)=z$.
By
\begin{align*}
W_{f,\varphi}W_{g,\psi}K_b=\overline{f(b)}W_{g,\psi}K_b
\end{align*}
and $W_{f,\varphi}e_0=\overline{f(b)} e_0$, we have
\begin{align*}
W_{g,\psi}K_b=e_0=e^{\bar{b}z-\frac{1}{2}|b|^2},
\end{align*}
so we get $g(z)=e^{-\frac{1}{2}|b|^2}$.

\begin{thm}\label{thm5.3}
Let $g\in H^\infty$ which has no zero on $\mathbb D$, and let ~$\psi$ be an analytic map of $\mathbb{D}$ into itself.Assume that $W_ {f,\varphi}$ is self-adjoint and $\varphi$, not an elliptic automorphism, has a fixed point $b$ $\mathbb{D}$. Then $W_{g,\psi}\in \{W_{f,\varphi}\}^{'}$ and $C_{\psi}$ is bounded on Fock space $F^2$ if and only if $W_{g,\psi}$ has the following symbol functions; 
$$ \psi(z)=z,\ \ g(z)=e^{-\frac{1}{2}|b|^2}.$$
\end{thm}
\begin{proof}
   We only need to show that $g(f\circ\psi)=f(g\circ\varphi)$ and $\varphi\circ\psi=\psi\circ\varphi$ hold and $C_{\psi}$ is bounded on $F^2$.
  First
$$\varphi(\psi(z))=\varphi(z),\psi(\varphi(z))=\varphi(z),$$
so $\varphi\circ\psi=\psi\circ\varphi$ holds.
Then 
$$g(z)\cdot f(\psi(z))=g(z)\cdot f(z),f(z)\cdot g(\varphi(z))=f(z)\cdot g(z)$$
so $g(f\circ\psi)=f(g\circ\varphi)$ also holds,
Therefore $W_{f,\varphi}W_{g,\psi}=W_{g,\psi}W_{f,\varphi}$,
that is $W_{g,\psi}\in \left\{W_{f,\varphi}\right\}^{'}$.
Finally to show $C_{\psi}$ is bounded on $F^2$.For $h\in F^2$
\begin{align*}
  \big\|C_{\psi}h\big\|_2^2 &= \frac{1}{\pi}\int_{\mathbb D}\big|C_{\psi}h(z)\big|^2e^{-|z|^2}dA(z) \\
  &=\frac{1}{\pi}\int_{\mathbb D}\big|h(\psi(z))\big|^2e^{-|\psi(z)|^2}e^{|\psi(z)|^2-|z|^2}dA(z)\\
  &\leq \sup_{z\in\mathbb D}\left\{e^{|\psi(z)|^2-|z|^2}\right\} \cdot \frac{1}{\pi}\int_{\mathbb D}\big|h(\psi(z))\big|^2e^{-|\psi(z)|^2}dA(z),
\end{align*}
Replacing $w$ by $\psi(z)$, then
\begin{align*}
  \big\|C_{\psi}h\big\|_2^2 &\leq \sup_{z\in\mathbb D}\left\{e^{|\psi(z)|^2-|z|^2}\right\} \cdot \frac{1}{\pi}\int_{\mathbb D}\big|h(w)\big|^2e^{-|w|^2}dA(w) \\
  &=\frac{C}{\pi}\int_{\mathbb D}\big|h(w)\big|^2e^{-|w|^2}dA(w)\\
  &= C\big\|h\big\|_2^2<\infty.
\end{align*}
Then we can get $C_{\psi}$ is bounded on $F^2$.
\end{proof}

Next we shoe that weighted composition operator $W_{g,\psi}$ in Theorem \ref{thm5.3} is normal.
\begin{prop}\label{prop5.1}
    Let $\varphi(z)=az+b$ be an analytic map of $\mathbb D$, then $C_{\varphi}^*=T_{K_b}C_{\overline{a}z }$. 
\end{prop}
\begin{proof}
  For $\beta\in \mathbb D$, by $C_{\varphi}^*K_{\beta}=K_{\varphi(\beta)}$,
  \begin{align*}
  T_{K_b}C_{\overline{a}z }K_{\beta}(z) &= K_b(z)\cdot K_{\beta}(\overline{a}z) \\
  &=e^{\overline{b}z+\overline{a\beta}z}\\
  &=K_{\varphi(\beta)}(z).
\end{align*}
Thus for every $h\in F^2$, we have $C_{\varphi}^*h=T_{K_b}C_{\overline{a}z }h$, therefore $C_{\varphi}^*=T_{K_b}C_{\overline{a}z }$.
\end{proof}

\begin{prop}\label{prop5.2}
    Let $\varphi(z)=az+b$ be an analytic map of $\mathbb D$, then $W_{\psi,\varphi}$ is normal if and only if $\overline{a}=1$ or $b=0$.
\end{prop}
\begin{proof}
    By
\begin{align*}
(W_{\psi,\varphi}^*W_{\psi,\varphi})f(z)&=C_{\varphi}^*M_{\psi}^*M_{\psi}C_{\varphi}f(z)\\
&=T_{K_b}C_{\overline{a}z }M_{\psi}^*M_{\psi}C_{\varphi}f(z)\\
&=\left | \psi \right | ^2K_bf\circ\varphi\circ \overline{a}z  \\
&=\left | \psi \right | ^2K_bf(|a|^2z+b)
\end{align*}
and
\begin{align*}
(W_{\psi,\varphi}W_{\psi,\varphi}^*)f(z)&=M_{\psi}C_{\varphi}C_{\varphi}^*M_{\psi}^*f(z)\\
&=M_{\psi}C_{\varphi}T_{K_b}C_{\overline{a}z }M_{\psi}^*f(z)\\
&=\left | \psi \right | ^2K_bf\circ\overline{a}z\circ\varphi  \\
&=\left | \psi \right | ^2K_bf(|a|^2z+\overline{a}b),
\end{align*}
we know that $W_{\psi,\varphi}$ is normal on $F^2$ if and only if
$$|a|^2z+b=|a|^2z+\overline{a}b.$$
So $\overline{a}=1$ or $b=0$.
\end{proof}

By proposition \ref{prop5.2}, we can show the weighted composition operator $W_{g,\psi}$ in theorem \ref{thm5.3} is normal .
\begin{prop}\label{prop5.3}
Let $g\in H^\infty$ which has no zero on $\mathbb D$, and let ~$\psi$ be an analytic map of $\mathbb{D}$ into itself. Assume that $W_ {f,\varphi}$ is self-adjoint and $W_{g,\psi}\in \left\{W_{f,\varphi}\right\}'$, and $\varphi$, not an elliptic automorphism, has a fixed point $b$ $\mathbb{D}$ and $C_{\psi}$ is bounded on $F^2$. Then $W_{g,\psi}$ is normal on $F^2$.
\end{prop}
\begin{proof}
    By theorem \ref{thm5.3}, $W_{g,\psi}$has the following symbol functions:
$$ \psi(z)=z,\ \ g(z)=e^{-\frac{1}{2}|b|^2}.$$
By proposition \ref{prop5.2}, $\psi(z)$ satisfies the condition, so $W_{g,\psi}$ is normal.
\end{proof}





\phantomsection

\bibliographystyle{unsrt}
\bibliography{sample.bib}


\end{document}